\documentclass[11pt,reqno]{amsart}
\usepackage{graphicx} % Required for inserting images
\usepackage{bbm}
\usepackage{enumerate}
\usepackage{mathtools}
\usepackage{comment}
\usepackage{mathrsfs}  
\usepackage{bbm}
\usepackage{ytableau}
\usepackage{tikz}
\usepackage{hyperref}
\hypersetup{
    colorlinks=true,
    citecolor=red,
    linkcolor=blue,
    urlcolor=red,
}
\usepackage[capitalize, nameinlink]{cleveref}
\usepackage[final]{showlabels}
\usepackage{amsfonts, amsmath, amssymb, amsthm}
\usepackage[backend=bibtex, giveninits=true, doi=false,isbn=false,url=false,style=alphabetic, maxnames=8, maxalphanames=9]{biblatex}
\bibliography{Ref.bib}
\numberwithin{equation}{subsection}
\usepackage[margin = 0.9in]{geometry}
\newtheorem{theorem}{Theorem}[subsection]
\newtheorem{prop}[theorem]{Proposition}
\newtheorem{lemma}[theorem]{Lemma}
\newtheorem{corollary}[theorem]{Corollary}

\theoremstyle{definition}
\newtheorem{remark}[theorem]{Remark}
\newtheorem*{example}{Example}
\theoremstyle{definition}
\newtheorem{definition}[theorem]{Definition}
\newcommand{\paren}[1]{\left( #1 \right)}
\newcommand{\set}[1]{\left\{#1\right\}}
\newcommand{\abrac}[1]{\left<#1\right>}
\newcommand{\sbrac}[1]{\left[#1 \right]}
\newcommand{\xequal}[1]{\overset{#1}{=\joinrel=}}
\newcommand{\xequals}[1]{\overset{#1}{=\joinrel=\joinrel=}}
\newcommand{\kb}{\mathbbm{k}}
\newcommand{\sym}{\mathrm{Sym}}
\newcommand{\ind}{\mathrm{IND}}
\newcommand{\res}{\mathrm{RES}}
\newcommand{\Zb}{\mathbb{Z}}

\newcommand{\nh}{\mathrm{NH}}
\newcommand{\As}{\mathscr{A}}
\newcommand\scalemath[2]{\scalebox{#1}{\mbox{\ensuremath{\displaystyle #2}}}}
\newcommand{\define}[4]{\expandafter#1\csname#3#4\endcsname{#2{#4}}}
\forcsvlist{\define{\DeclareMathOperator}{}{}}{ad,Ad,codim,GL,SL,vdim} %operators \bla
\forcsvlist{\define{\newcommand}{\mathbf}{bf}}{A,C,D,H,S,T,v,w,Z,a,b,f,i,j,k,1,K} %mathbf letters \bfX
\forcsvlist{\define{\newcommand}{\mathbb}{bb}}{A,D,H,k,N,Z,F,Q,R,P,C,T,L,G} %mathbb letters \bbX
\forcsvlist{\define{\newcommand}{\mathcal}{c}}{A,B,C,D,E,F,G,H,I,J,K,L,M,N,O,P,Q,R,S,T,U,V,W,X,Y,Z} %mathcal letters \cX
\forcsvlist{\define{\newcommand}{\mathfrak}{fk}}{a,h,g,gl,H,I,S,sl,z} %mathfrak letters \fkX
\forcsvlist{\define{\newcommand}{\mathrm}{}}{BM,crit,id,Id,pt,op,sm,gdim} %mathrm words \bla
\forcsvlist{\define{\DeclareMathOperator}{\mathsf}{}}{Ann,Aut,Bun,cl,Co,Coh,Coker,colim,Cone,Crit,End,Ext,Fl,Frac,Fun,gr,Gr,Graph,Hecke,Higgs,Hom,Ind,IND,Jac,Ker,loc,LSym,LPoly,Mat,mon,NC,Nil,Pol,Proj,Quot,Rep,Res,RHom,rk,Sch,Spec,Stab,Supp,Sym,Tor,Tot,Tr,tw,Vect,TRes, I, TR} %mathsf words \bla
\renewcommand{\mod}{{\mathrm{mod}}}
\renewcommand{\pmod}{{\mathrm{pmod}}}
\newcommand{\fmod}{{\mathrm{fmod}}}
\usepackage[shortlabels]{enumitem}
\newenvironment{enua}{\begin{enumerate}[label=\textup{(\alph*)}]
}{\end{enumerate}}

\def\<{\langle}%
\def\>{\rangle}%
\def\-{\text{-}}%
\newcommand{\inv}{^{-1}}%
\newcommand{\tbinomb}[2]{\genfrac{[}{]}{0pt}{}{#1}{#2}}
\title[Infinite-Dimensional Towers]{Infinite-Dimensional Towers and a Categorification of Differential Operators on $\mathbb{A}^1_{\mathbb{Z}[v,v^{- 1}]}$}

\begin{document}
\author{\sc Chun-Ju Lai}
\address{Institute of Mathematics\\ Academia Sinica \\Taipei\\ 106319, Taiwan}
\thanks{Research of the first author was supported in part by
NSTC grants 113-2628-M-001-011 and the National Center of Theoretical Sciences}
\email{cjlai@gate.sinica.edu.tw}

\author{\sc Cailan Li}
\address{Institute of Mathematics\\ Academia Sinica \\Taipei\\ 106319, Taiwan}
\email{cli@gate.sinica.edu.tw}

\begin{abstract}
We introduce axioms for towers of infinite-dimensional algebras such that the corresponding Grothendieck groups of projective and finite-dimensional modules are Hopf dual to each other. This duality gives rise to an action of the Hesienberg double---a generalization of the classical Hesienberg algebra---on the Grothendieck group. We categorify this action and, as an application, construct a categorical realization of quantum differential operators on $\bbA^1_{\Zb[v,v^{-1}]}$.
\end{abstract}

\maketitle

\section{Introduction}
\subsection{}
%TODO, explain hesienberg cat in 2 steps, computation of ring structure on K(A)/G(A), and then double?
For $n\geq 0$, let $A_n$ be a unital associative algebra over a field $\kb$. Let $A_n\-\fmod$ and $A_n\-\pmod$ be the categories of finite-dimensional and finitely generated projective $A_n$-modules, respectively, and let $\cK(A)$ and $\cG(A)$ be the corresponding Grothendieck groups. Following Bergeron and Li \cite{BL09}, a \emph{tower of algebras} refers to a direct sum $A = \bigoplus_{n \geq 0} A_n$ of certain finite-dimensional algebras $A_n$ satisfying certain axioms. These axioms ensure that $\cK(A)$ and $\cG(A)$ admit the structure of a graded connected bialgebra (and thus a Hopf algebra) that are Hopf dual to each other with respect to the $\Hom$ form:
\begin{equation*}
\cK(A) \times \cG(A) \to \bbZ,
\quad([P], [M]) \mapsto \dim_\kb \Hom_A(P,M).
\end{equation*}
Given an element $[P]\in \cK(A)$, this Hopf duality allows one to define $r_P^*$, the adjoint of right multiplication by $[P]$.
It further allows one to define the Heisenberg double (see \cite{L94}) of $\cG(A)$ as follows:
%TODO, when defining heisenberg double later, state how this def is equivalent b/c of faithfulness of fock space
\[
\fkh(A) %=\fkh(\cG(A)) 
:= \< \ell_{[M]}, r^*_{[P]} ~|~ M \in A_n\-\fmod, \ P \in  A_n\-\pmod, \ n\geq 0\> \subseteq \End_\kb (\cG(A)),
\]
where $\ell_M$ is the left multiplication by $[M]$.
The $\fkh(A)$-module $\cG(A)$ is called the {\em Fock space} representation of $\fkh(A)$. In particular, the representation $\fkh(A) \curvearrowright \cG(A) $ recovers the following important examples in representation theory:
\begin{enumerate}
    \item The polynomial representation of the Weyl algebra $W$, i.e.,
\[
\fkh(A) \cong W := \< x, \partial\> \subseteq \End_\kb (\bbZ[x]),
\]
where $A_n := \NC_n$ is the nilCoxeter algebra introduced by Fomin and Stanley (see \cite{FS94}). 
    \item The Heisenberg action on symmetric functions, i.e.,
\[
\fkh(A) \cong \mathsf{Heis} \subseteq \End_\kb(\Sym),
\]
where $A_n := \bbC[S_n]$ is the group algebra of the symmetric group $S_n$.
 Here, $\Sym$ is the ring of symmetric functions in infinitely many variables.
\end{enumerate}
These Fock space representations were categorified by Khovanov in \cite{K01,Kho14}, and also Savage and Yacobi in \cite{SY15}, in the sense that the relations satisfied by the endomorphisms $\ell_M$ and $r^*_P$ acting on $\cG(A) $ lift to corresponding relations between certain functors acting on $ \bigoplus_{n \geq 0} A_n\-\fmod$. For example, the relations $[\partial, x] = 1$ in $W$ and $[h_i^*,e_j] = e_{j-1}h^*_{i-1} $ in $\mathsf{Heis}$ translate to, respectively: 
\[
\Res\Ind \cong \Ind\Res \oplus \Id,
\quad
\Res_{H_i}\Ind_{E_j} \cong \Ind_{E_j}\Res_{H_i} \oplus \Ind_{E_{j-1}}\Res_{H_{i-1}}, 
\]
where $\Res=\Res_{\NC_n}^{\NC_{n+1}}$, \vspace{0.5ex} $\Ind=\Ind_{\NC_n}^{\NC_{n+1}}$ and $H_j, E_j$ are the trivial and sign representations of $S_j$, respectively (see \cite{SY15} for the definition of $\Ind_{E_j}$, etc). In particular, given a tower $A$ of finite-dimensional algebras  satisfying certain conditions, Savage and Yacobi developed a framework for the categorification of the Fock space of $\fkh(A)$ that unifies both cases above.

\subsection{}
In this paper, we develop a framework for the categorification of the Fock space of $\fkh(\As)$, where the tower $\As=\bigoplus_{n \ge 0} A_n$ is built from infinite-dimensional algebras $A_n$. When the algebra $A_n$ is finite-dimensional, there is a well known canonical bijection
\begin{equation}\label{eq:bij}
\{ \text{indecomposable projective $A_n$-modules} \} / \cong 
\;\;\longleftrightarrow\;\;
\{ \text{simple $A_n$-modules} \} / \cong,
\end{equation}
which gives the duality between $\cK(A)$ and $\cG(A)$. When the algebra $A_n$ is infinite-dimensional, this correspondence no longer holds. For example, the polynomial algebra $A_n=\kb[x_1, \ldots, x_n]$ has one indecomposable projective module, but infinitely many simples. However, if one endows $A_n$ with the grading $|x_i|=1$, and pass to graded modules, then this issue disappears, as $\kb[x_1, \ldots, x_n]$ only has one graded simple, namely $\kb$. 
Thus, we will work with infinite-dimensional algebras $A_n$ that are \emph{graded}. 

To ensure the corresponding bijection \eqref{eq:bij} in the graded setting, we also require the algebras be \emph{nice} (see \cref{nicedef}). Under additional hypotheses (see \cref{towerdef}), we prove that $\cK(\As)$ and $\cG^f(\As)$ are Hopf-dual algebra–coalgebras and given $P \in A_n\-\pmod$, $M \in A_n\-\fmod$, we define functors
\[  
\I_M: \As\-\fmod\to \As\-\fmod 
\quad\textup{and}\quad 
\TR_P: \As\-\fmod \to \As\-\fmod 
\]
that categorify $r_{[M]}$ and $r_{[P]}^*$, respectively. 

Examples of (infinite-dimensional) nice graded algebras include the polynomial ring $\kb[x_1, \ldots, x_n]$, 
as well as special instances of quantum wreath products $B \wr \cH(n)$ introduced in \cite{LNX24}, 
in the case the base algebra $B$ is infinite-dimensional. 
A notable example is the nilHecke algebra $\nh_n$. Applying our framework to the nilhecke tower $\bigoplus_{n\ge 0}\nh_n$, 
we thereby obtain a categorification of the action of quantum differential operators on $\bbA^1_{\Zb[v,v^{-1}]}$. 
More concretely, 
recall the quantum integer $(a)_v:= \frac{v^a-1}{v-1}$ and quantum binomial coefficients $\tbinom{a}{b}_v := \prod_{i=0}^{b-1}\frac{(a-i)_v}{(b-i)_v}$,
we categorify the action of $\cD$ on $\Zb[v^{\pm 1}][x]$, where
\[
\cD :=
\frac{\Zb[v^{\pm 1}]\langle x, \partial^{(n)} | \, n \ge 0 \rangle}
{(
%x^nx^m-x^{n+m}, \;\;
\partial^{(m)} \partial^{(n)} - \tbinom{m+n}{n}_v\partial^{(m+n)}, 
\quad
\partial^{(m)} x - v^m x \partial^{(m)} - \partial^{(m-1)}
~|~
m,n \geq 0
)}.
\]
Here, $x\in \cD$ acts on $\Zb[v^{\pm 1}][x]$ by right multiplication, while $\partial^{(m)}\in \cD$ acts on $\Zb[v^{\pm 1}][x]$ by 
\[ 
\partial^{(m)}(x^n)=\tbinom{n}{m}_v x^{n-m}.
\]

\section{Infinite-dimensional Towers of Algebras}
\subsection{Graded Simple Modules}
We begin with an analog of \eqref{eq:bij} in the graded setting. 

\begin{lemma}
\label{gradsimplem}
    Any graded simple $S$ of a positively %(or, negatively) 
    graded algebra $R=\oplus_{i\ge 0} R_i$  is a simple module over $R_0:= R/R^+$.
\end{lemma}
\begin{proof}
    Since $R^+S\subseteq S$ is a graded $R$-submodule, either $R^+S=0$ or $R^+S=S$. The latter is impossible, as $R^+$ is positively graded. 
\end{proof}

\begin{definition}
\label{nicedef}
A $\Zb$-graded algebra $R=\oplus_{i\in \Zb} \, R_i$ is called {\em nice} if the following conditions hold:
\begin{enumerate}
        \item There exists a positively (negatively) graded ideal $I\subseteq \oplus_{i> 0} R_i$, such that all graded simples of $R$ are annihilated by $I$.
        \item The quotient $R/I$ is a finite-dimensional algebra.
    \end{enumerate}
\end{definition}

\begin{prop}
\label{niceprop}
    Let $R$ be a nice graded algebra. Then there is a bijection 
    \[ 
    \left\{
    \begin{tabular}{c}
    graded finitely generated \\
    indecomposable projectives of $R$
    \end{tabular}
    \right\}
    /\textnormal{shift } \longleftrightarrow \set{\textnormal{graded simples of }R }/\textnormal{shift }.  
    \]
   Index either set by $J$. If $\kb$ is algebraically closed, then $\dim_\kb (P_i, S_j)=\delta_{ij}$ for $i\in J$ where the indecomposable projective $P_i$ corresponds to the graded simple $S_i$ in the bijection above.
\end{prop}
\begin{proof}
    Given a simple $S_b$ of $R$, $(1)$ says that $S_b$ is a simple for $R/I$ which by $(2)$ is a finite-dimensional algebra and thus we have the bijection there. 
    
    Given the projective cover $\psi: \overline{P_b}\twoheadrightarrow S_b$ we have that $P_b$ is a principal indecomposable and so $\overline{P_b}=R/I \cdot \overline{e}$ for some idempotent $\overline{e}\in R/I$. But $I$ is positively graded and thus $\overline{e}=e\in B$ already. Thus we can construct
    \[ 
    P_b:=R e\xrightarrow{\pi} \overline{P_b}\xrightarrow{\psi} S_b.
    \]
   $IP_b\subset J_{gr}(R) P_b$ where $J_{gr}(R)$ is the graded jacobson radical of $R$. Nakayama implies that $J_{gr}(R) P_b$ is a superfluous submodule of $P_b$ and thus $IP_b=\ker \pi$ is a superfluous submodule. $\pi$ is a superfluous epimorphism and since the composition of superfluous epimorphisms is a superfluous epimorphism we see that $P_b$ is the projective cover of $S_b$ as an $R$ module. It follows that $R$ is a semiperfect ring and semiperfect rings afford the bijection above along with the calculation of $\dim_\kb (P_i, S_j)$.
\end{proof}

\subsection{Tower of Algebras}
\label{towersect}
Given a graded $\kb$-algebra $R$, we will, by a slight abuse of notation, write
$R\text{-}\fmod$ and $R\text{-}\pmod$ for the categories of \emph{graded} finite-dimensional $R$-modules and \emph{graded} finitely generated projective $R$-modules, respectively.
We further define
\[
    G_0^f(R):=K_0(R\-\fmod),
    \quad
    K_\oplus(R):=K_0(R\-\pmod).
\]
Let $F:R\-\mod \to R^{op}\-\mod$ be a graded additive functor sending projective modules to flat modules. Then there is a bilinear pairing
\begin{equation}\label{paireq}
\abrac{-,-}_F:K_\oplus(R) \times G_0^f(R) \to \Zb[q, q^{-1}], 
\quad 
\abrac{[P], [M]}:=\mathrm{gdim}_\kb F(P)\otimes_{R} M.
\end{equation}

\begin{remark}
The form $\paren{-,-}_{\Hom}=\dim_\kb \Hom_R(P, M)$ used in \cite{SY15} is a special case of the above as follows. 
Let $R$ be a Frobenius algebra (which is the case for the examples in \cite{SY15}).  Define $F(-):=\Hom_R(-, R)$. 
Since Frobenius algebras are self-injective, $F$ sends projectives to injectives, which coincides with the projectives in $R^{op}$, see \cite[Chapter IV, Proposition 3.7]{SY11}. 
With this $F$, we then have $\abrac{-,-}_F=\paren{-,-}_{\Hom}$ because
   \[ 
   F(P)\otimes_R M=\Hom_R(P, R)\otimes_R M =\Hom_R(P, M), 
   \]
  where the second equality follows from the projectivity of $P$. 
\end{remark}

%In the following, we extend the definition of the tower of algebras in \cite{SY15}.

\begin{definition}
\label{towerdef}
  A graded $\kb$-algebra $\As=\oplus_{n\geq 0} A_n$ is called a \emph{nice tower of algebras} if the following are satisfied for all $m,n$:
  \begin{enumerate}
      \item[(TA1)] $A_n$ is a nice graded algebra. Moreover, $A_0=\kb$.
      \item[(TA2)] There is a unital algebra morphism $A_n\otimes A_m\to A_{n+m}$.
      \item[(TA3)] $A_{m+n}$ is a two-sided projective $A_n\otimes A_m$ module.
      \item[(TA4)] There exists an anti-automorphism $\psi$ of $A_n$.
  \end{enumerate}
\end{definition}
%[TODO, need to add absolutely irreducible condition. ][TODO, explain how finite-dimensional algebras fit into picture]

\begin{comment}
Let $S_\lambda := S_{\lambda_1} \times \dots \times S_{\lambda_r} \subseteq S_{n+m}$ be the Young subgroup with respect to the composition $\lambda = (\lambda_1, \dots, \lambda_r) \vDash n+m$. 
Denote the set of minimal length left coset representatives for $S_\lambda \backslash S_{n+m}$ by $\cD\inv_\lambda$.

Since the defining relations of a quantum wreath product are local, (TA2) is satisfied. 
By \cref{freelem}, any (graded) quantum wreath product $B\wr \cH(n)$ with a PBW basis satisfies (TA3). 
Viewing elements of $B\wr \cH(n)$ as crossings decorated by elements of $B$, there is an anti-automorphism $\psi$ of $B\wr \cH(n)$ given by flipping diagrams horizontally. Finally, below are examples where (TA1) is also satisfied  

\begin{example}
\begin{enumerate}
    \item $NH_m$ is a nice graded algebra where $I=(\mathrm{Sym}_m^+)$.
    \item By \cref{gradsimplem}, $nNH_m$ is a nice graded algebra where $I=nNH_m^-$. 
    \item By \cref{gradsimplem}, $CNH_m$ is a nice graded algebra where $I=CNH_n^+$. 
\end{enumerate}    
\end{example}

\end{comment}

Given a nice tower of algebras $\As$, define
\[
\cG^f(\As):=\bigoplus\limits_{n \geq 0} G_0^f(A_n),
\quad\textup{and}\quad
\cK(\As):=\bigoplus\limits_{n\geq 0} K_\oplus(A_n).
\]
%TODO, add in prelim that all modules are left so R^{op} is the same as right modules. 
Let $F: A_n\-\mod \to A_n^{op}\-\mod$ be the graded additive functor fixing morphisms and objects, where the right $A_n$-action is given by $m\cdot^{\psi} a :=\psi(a)\cdot m$. 
The functor $F$ sends projectives to projectives, so from \cref{paireq} we have a pairing 
\begin{equation}\label{def:<>psi}
\abrac{-,-}_\psi:K_\oplus(A_n)\times G_0^f(A_n)\to \Zb[q, q^{-1}],
\quad
\abrac{[P],[N]}_\psi:=\mathrm{gdim}_\kb P^\psi \otimes_{A_n} N.
\end{equation} 
\begin{prop}
\label{perfectlem}
Given a nice tower of algebras $\As$, the pairing $\abrac{-,-}:\cK(\As) \times \cG^f(\As) \to \Zb[q, q^{-1}]$ defined by
\begin{equation}\label{def:pairingKGf}
\abrac{[P], [M]}:=\begin{cases}
    \abrac{[P], [M]}_\psi & \textnormal{if }P\in A_n\-\pmod \textnormal{ and }M\in A_n\-\fmod; 
    \\
    0 & \textnormal{otherwise.}
\end{cases} 
\end{equation}
is non-degenerate. If additionally, $\kb$ is algebraically closed, then the pairing is perfect. 
\end{prop}
\begin{proof}
As $G_0^f(A_n)$ and $K_\oplus(A_m) $ are orthogonal if $n\neq m$, 
we just need to show non-degeneracy for $n=m$. 
Note that $K_\oplus(A_m) $ has a $\Zb[q, q^{-1}]$-basis of indecomposable projectives. 
(TA1) and \cref{niceprop} imply all indecomposable projectives of $A_n$ are projective covers of simples $S_j$, and are principal, i.e. of the form $A_n e_j$ for some idempotent $e_j$. 
As right $A_n$-modules, we have $(A_n e)^\psi\cong\psi(e) A_n$ sending $be\mapsto \psi(e) \psi(b)$. 
Thus
\[ 
\abrac{[A_n e_i], [S_j] }_\psi=\gdim_\kb (\psi(e_i) A_n\otimes_{A_n} S_j)=\gdim_\kb \psi(e_i) S_j=\gdim_\kb \Hom_{A_n}(A_n \psi(e_i) , S_j). 
\]
It follows that $\set{[A_n\psi(e_i)]}$ and $\set{[S_j]}$ are orthogonal and if $\kb$ is algebraically closed they are dual basis from \cref{niceprop}.
\end{proof}
\begin{remark}
    The pairing \eqref{def:pairingKGf} coincides with the bilinear pairing defined for KLR algebras in \cite{KL09}.
\end{remark}
Using (TA3) the functors below induce maps $\sbrac{\ind}, \sbrac{\eta}$ on $\cG^f(\As)$ and $\cK(\As)$. 
\begin{align*}
    &\ind: \As^{\otimes 2}-\mathrm{mod}\to \As-\mathrm{mod} &\eta: \kb-\mathrm{mod}\to \As-\mathrm{mod}\phantom{,}\\
    &\ind|_{A_m\otimes A_n-\mod}=\Ind_{A_m\otimes A_n}^{A_{m+n}} &\eta(V)=V \textnormal{ for }V\in \kb-\mod.
\end{align*}
\begin{comment}
\begin{align*}
    \begin{aligned}       
    \ind&: \As^{\otimes 2}\-\mod\to \As\-\mod, 
    \\
    \ind&|_{A_m\otimes A_n\-\mod}:=\Ind_{A_m\otimes A_n}^{A_{m+n}},     
    \end{aligned}
    &&
    \begin{aligned}
    \eta&: \kb\-\mod\to \As\-\mod,
    \\
    \eta&(V):=V \textnormal{ for }V\in \kb\-\mod.
    \end{aligned}    
\end{align*}
\end{comment}

\begin{lemma}
    $\sbrac{\ind}$ turns $\cG^f(\As)$ and $\cK(\As)$ into unital associative algebras with unit $\sbrac{\eta}$.  
\end{lemma}
Similarly, the functors below induce maps  $\sbrac{\res}, \sbrac{\epsilon}$ on $\cG^f(\As)$ and $\cK(\As)$.
\begin{comment}
\begin{align*}
\begin{aligned}
    \res&:\As\-\mod\to \As^{\otimes 2}\-\mod,
    \\
    \res&|_{A_n\-\mod}:=\bigoplus\nolimits_{k+\ell=n}\Res_{A_k\otimes A_\ell}^{A_{n}}, 
\end{aligned}
&&
\begin{aligned}
    \epsilon&: \As\-\mod\to \kb\-\mod,
    \\    
    \epsilon&(V):=\begin{cases}
        V &\textnormal{ if }V\in \kb\-\mod;
        \\
        0 & \textnormal{ otherwise.}
    \end{cases}
\end{aligned}
\end{align*}
\end{comment}
\begin{align*}
    &\res: \As-\mathrm{mod}\to \As^{\otimes 2}-\mathrm{mod}&\epsilon: \As-\mathrm{mod}\to \kb-\mathrm{mod}\phantom{,}\\
    &\res|_{A_n-\mod}=\bigoplus_{k+\ell=n}\Res_{A_k\otimes A_\ell}^{A_{n}} &\epsilon(V)=\begin{cases}
        V &\textnormal{ if }V\in \kb-\mod, \\
       0 & \textnormal{ otherwise }
    \end{cases}
\end{align*}
\begin{lemma}
    $\sbrac{\res}$ turns $\cG^f(\As)$ and $\cK(\As)$ into counital coassociative coalgebras with counit $\sbrac{\epsilon}$.  
\end{lemma}

\subsection{The Bialgebra pairing}
%Now, We adopt the sumless Sweedler notation   $\Delta(x) = x_{(1)} \otimes x_{(2)}$.
For the rest of the article, we abbreviate $\boxtimes=\otimes_\kb$.

\begin{definition}
A pairing $\abrac{-,-}: H\times K\to \Zb[q,q^{-1}]$ between two bialgebras $H$ and $K$ is a bialgebra pairing if the following hold, for $p, p' \in H$, $m,m' \in K$: 
%\vspace{-1.5ex} 
\begin{align}
\label{dual1}
\<pp^\prime,m\> = \<p\boxtimes p', \Delta_K(m)\>,
%\<[P],[M]_{(1)}\>\<[P'],[M]_{(2)}\>,    
\\
\label{dual2}
\<p,m m^\prime\> = \<\Delta_H(p), m \boxtimes m'\>,
%\<[P]_{(1)},[M]\>\<[P']_{(2)},[M]\>,    
\\
\label{dual3}
\< 1_H,m\> = \epsilon_K(m),
\quad\< p,1_K\> = \epsilon_H(p).
\end{align}    
where $\abrac{-,-}$ on the RHS of \eqref{dual1} and \eqref{dual2} is given by
\[
\label{tensorpair}
    \abrac{p\boxtimes p^\prime, k \boxtimes k^\prime}=\abrac{p,k}\abrac{p^\prime, k^\prime}.
\]
\end{definition}

\begin{remark}
    Additionally, if $H, K$ are connected Hopf algebras, then a bialgebra pairing is automatically a Hopf pairing. 
\end{remark}

\begin{prop}
    Given a nice tower of algebras $\As$, the pairing $\abrac{-,-}:\cK(\As) \times \cG^f(\As) \to \Zb[q, q^{-1}]$ from \eqref{def:pairingKGf} is a bialgebra pairing. 
\end{prop}
\begin{proof}    
 First note that \eqref{tensorpair} for $\cK(\As) \times \cG^f(\As)$ can be rephrased as
\[ \abrac{ [P]\boxtimes [P'], [M]\boxtimes[M'] }=\gdim_\kb P^\psi \boxtimes P'^\psi \otimes_{A_n\boxtimes A_m} M\boxtimes M'. 
\]
%WLOG, $P_1, P_2$ are indecomposable projectives,
Now we show that \eqref{dual1} holds. Given $P\in A_n\-\pmod$, $P'\in A_m\-\pmod$, and $M\in A_{n+m}\-\fmod$, we may assume that both $P$ and $P'$ are indecomposable.
It follows from (TA1) and \cref{niceprop} that $P=A_n e, P'=A_m e'$ for some $e$, $e'$. Then,
\begin{equation}\label{spliteq}
\begin{split}
    \Big( 
    A_{n+m} 
    \mathop\otimes\limits_{A_n\boxtimes A_m} 
    P\boxtimes P' 
    \Big)^\psi
    &=
    \Big(
    A_{n+m}
    \mathop\otimes\limits_{A_n\boxtimes A_m} 
    (A_n\boxtimes A_m)e\boxtimes e'
    \Big)^\psi 
    \\
    &=(\psi(e)\boxtimes \psi(e')) A_{n+m}
    = (P^\psi \boxtimes P'{}^\psi) \otimes_{A_n\boxtimes A_m} A_{n+m}.
\end{split}
\end{equation}
Thus,
\begin{align*}
    \abrac{[\ind](P\boxtimes P'), M}
    \xequal{\eqref{spliteq}}
    \gdim_\kb \big(
    (P^\psi \boxtimes P'{}^\psi) 
    \mathop\otimes\limits_{A_n\boxtimes A_m} 
    A_{n+m} \mathop\otimes\limits_{A_{n+m}} M 
    \big) 
    =\abrac{P\boxtimes P', [\res](M)},
\end{align*}
where we have used that the pairing vanishes on $\Res^{A_{n+m}}_{A_k\otimes A_\ell}$ if $k\neq m, \ell\neq n$. 
Equation \eqref{dual2} proceeds similarly. 
Lastly, \eqref{dual3} follows directly from definition.
\end{proof}

\begin{remark}
    If one uses $\abrac{-,-}_{\Hom}$ instead, then the
    verification of \eqref{dual2} % other direction above 
    does not proceed similarly.
    In order to ensure the bialgebra pairing structure,
    one needs to impose different conditions on $A_n$ (such as being a Frobenius algebra).  
    For details, see \cite[Proposition 3.7]{SY15} and \cite[Proposition 3.3]{BL09}.
\end{remark}

\begin{comment}

\subsection{Mackey Filtration}
Let $n,m,k,\ell$ be non-negative integers such that $n+m=k+\ell=N$. Denote by $_{k, \ell}(C_N)_{n,m}$ the natural $(C_{k,\ell}, C_{n,m})$-bimodule whose underlying space is $C_N$.

\begin{prop}[``Mackey filtration"]
Suppose $C_N:=B\wr \cH(N)$ is graded and has a PBW basis. Set $J=\mathrm{max}(0,n-\ell)$. Then, $_{k, \ell}(C_N)_{n,m}$ has a filtration 
    \[0=F_{\mathrm{min}(n,k)+1}\subseteq F_{\mathrm{min}(n,k)}\subseteq F_{\mathrm{min}(n,k)-1}\subseteq \ldots\subseteq F_{J}=_{k, \ell}(C_N)_{n,m} \]
such that each subquotient admits an isomorphism of graded $(C_k\boxtimes C_\ell, C_n\boxtimes C_m ) $ bimodules: 
\[ \Phi_r: F_r /F_{r+1} \xrightarrow{\sim}  q^{|H|(n-r)(k-r)} (C_k\boxtimes C_\ell) \otimes_{C_r\boxtimes C_{n-r}\boxtimes C_{k-r} \boxtimes  C_{\ell+r-n}} (C_n\boxtimes C_m ) \]
Here, $(qV)_a=V_{a-1}$, $|H|=\mathrm{deg}(H)$, and the right action of $C_r\boxtimes C_{n-r}\boxtimes C_{k-r} \boxtimes  C_{\ell+r-n}$  on $C_k\boxtimes C_\ell$ is given by
\[  
(c_1\otimes c_2\otimes c_3 \otimes c_4)\cdot(a\otimes b)=(c_1c_3 \cdot a) \otimes (c_2c_4\cdot b ). 
\]
\end{prop}

%\begin{proof}[TODO]
 % The set of minimal coset representatives of $S_k \times S_\ell \backslash S_N /S_n \times  S_m$ is given by $\set{w_{\mathrm{min}(n,k} \prec }$
%\end{proof}

\end{comment}
\subsection{Categorification of Fock Space}

% [need dualizing for \cref{actionlem} ]
For $M \in A_m\-\fmod$, $m \in \mathbb{N}$, define the functor $\I_M: \As\-\fmod\to \As\-\fmod$
by
\[
  \I_M|_{A_n\-\fmod}(N) \;=\;
  \Ind^{A_{n+m}}_{A_n \otimes A_m}\!\bigl(  N \boxtimes  M \bigr)
  \;\in\; A_{m+n}\-\fmod,
  \qquad
  N \in A_n\-\fmod,\; n \in \mathbb{N}.
\]
For each $P \in A_p\-\pmod$, $p \in \mathbb{N}$, define the functor
$\TR_P  : \; \As\-\fmod \to \As\-\fmod
$
by
\[
  \TR_P|_{A_n\-\fmod}(N)
  \;=\; P^\psi \otimes_{A_p} \Res^{A_n}_{A_{\,n-p} \otimes A_p} \! N
  \;\in\; A_{\,n-p}\-\fmod,
  \qquad
  N \in A_n\-\fmod,\; n \in \mathbb{N},
\]
where $\TR_P(N)$ is interpreted as the zero object of $\As\-\fmod$ whenever $n-p<0$. 
\begin{remark}
    The $A_p$-module structure on $\Res^{A_n}_{A_{\,n-p} \otimes A_p}$ comes from the embedding $A_p\hookrightarrow A_{\,n-p} \otimes A_p$, $a\mapsto 1\otimes a$, while the $A_{n-p}$-module structure on $\TR_P(N)$ is given by
\begin{equation} % boxed -> has an equation number
a\cdot (x\otimes n) \;:=\; x\otimes \bigl((a\otimes 1)\cdot n\bigr),
\qquad a\in A_{n-p},\; x\in P^{\psi},\; n\in N.
\end{equation}
\end{remark}
If $P=A_pe$ is a principal projective, then 
\[
\TR_{A_pe}(N)=\Hom_{A_p}
(A_p \psi(e), \Res^{A_n}_{A_{n-p}\otimes A_p} \, N  )= \Res_{A_p \psi(e)}(N), 
\]
where $\Res_P$ is the functor defined in \cite[Section 3.4]{SY15}. Both $\I_M$ and $\TR_P$ are exact and thus induce maps $[\I_M], [\TR_P]$ on $\cG^f(\As)$. 
Note that $\I_M$ categorifies the operator of right multiplication by $[M]$.
%(Mention the same with $[Res]$ but after moving contents) 
\begin{lemma}
For $M \in \As\-\fmod$ and $N\in \As\-\fmod$, 
$[\I_M](N)=[\ind](N\boxtimes M).$
\end{lemma}
In contrast, the situation for $\TR_P$ is less clear, but we claim that it categorifies the adjoint of right multiplication by $[P]$. 
\begin{prop} \label{actionlem}
For $P\in A_p\-\pmod$ and $N\in A_n\-\fmod$, 
\[
  r_{[P]}^*\paren{[N]}
  \;=\;
  \begin{cases}
    0, & \text{if } p>n, \\[6pt]
    \bigl[\TR_P(N)],
      & \text{if } p\le n.
  \end{cases}
\]
\end{prop}
\begin{proof}
    We proceed similarly to \cite[Lemma 3.10]{SY15}. For $P_1\in A_{n-p}\-\pmod$,
\[
    \begin{aligned}        
       %\scalemath{0.9}{
       \<[P_1], r_{[P]}^*\paren{[N] }\> %}
       &:= %\scalemath{0.9}{
       \abrac{[P_1] [P], [N] }
       =\abrac{ P_1\boxtimes P, [\res](N)} 
       \\
       &=\gdim_\kb (P_1^\psi \boxtimes P^\psi) \otimes_{A_{n-p}\boxtimes A_p}
       \Res^{A_n}_{A_{n-p}\otimes A_p}(N) 
       %}
       \\
        &\xequals{\cref{tendist}}
        \gdim_\kb 
        P_1^\psi  \otimes_{A_{n-p}} 
        P^\psi\otimes_{A_p} 
        \Res^{A_n}_{A_{n-p}\otimes A_p}(N)
        =\abrac{[P_1] , [\TR_P(N)] }.
    \end{aligned}        
\]
    Since $\abrac{-,-}$ is non-degenerate, the proof concludes.
\end{proof}

\begin{lemma}
\label{tendist}
    Let $S,R$ be $\kb$-algebras, $L$ a right $S$-module, $M$ a right $R$-module and $N$ a left $(S\otimes R)$-module. Then, there is a canonical isomorphism
    $$\paren{L\boxtimes M}\otimes_{S\otimes R} N \cong L\otimes_S \paren{M\otimes_R N}.$$
\end{lemma}

\begin{theorem}
For all $M,N\in \As\-\fmod$ and $P,Q\in \As\-\pmod$, we have the following isomorphisms of functors:
\begin{align}
\label{indeq}
    \I_M\circ \I_N&\cong \I_{\ind(N\boxtimes M)},
    \\ 
\label{reseq}
    \TR_P\circ \TR_Q&\cong \TR_{\ind(P\boxtimes Q)}.
\end{align}
\end{theorem}
\begin{proof}
  The proof of \eqref{indeq} is very similar to \cite[Theorem 3.18]{SY15} but since $\I_M$ places $M$ on the right, the order reverses. 
  For \eqref{reseq}, let $P\in A_p\-\pmod$, $Q\in A_q\-\pmod$ be principal indecomposables, 
  and $L\in A_\ell\-\fmod$.
  We compute that
\begin{equation}\label{eq:TRPTRQ}  
\begin{aligned}
      \TR_P\circ \TR_Q(L)
      & =P^\psi\otimes_{A_p}\Res^{A_{\ell-q}}_{A_{\ell-q-p}\otimes A_p}
        (Q^\psi\otimes_{A_q} \Res^{A_\ell}_{A_{\ell-q}\otimes A_q} L)
      \\
      &= P^\psi\otimes_{A_p} 
        (Q^\psi\otimes_{A_q} \Res^{A_\ell}_{A_{\ell-q-p}\otimes A_p \otimes A_q} L )
      \\
      &\xequals{\cref{tendist}} 
      (P^\psi\boxtimes Q^\psi)     
          \otimes_{A_p\otimes A_q} 
      (\Res^{A_\ell}_{A_{\ell-q-p}\otimes A_p\otimes A_q} L)
      \\
      &=(P^\psi\boxtimes Q^\psi) 
          \otimes_{A_p\otimes A_q} 
          (\Res^{A_{\ell-p-q}\otimes A_{p+q}}_{A_{\ell-q-p}\otimes A_p \otimes A_q}
          \circ
          \Res^{A_\ell}_{A_{\ell-q-p}\otimes A_{p+q}} L )
      \\
      &= (P^\psi\boxtimes Q^\psi) 
      \otimes_{A_p\otimes A_q} A_{p+q}
      \otimes_{A_{p+q}} 
      \Res^{A_\ell}_{A_{\ell}\otimes A_{p+q}} L
      \\
      &\xequals{\eqref{spliteq}} \ind(P\boxtimes Q)^\psi \otimes_{A_{p+q}} \Res^{A_\ell}_{A_{\ell-q-p}\otimes A_{p+q}} L 
      =\TR_{\ind(P\boxtimes Q)}.
  \end{aligned}
\end{equation}
Because $\TR_{P_1\oplus P_2}=\TR_{P_1}\oplus \TR_{P_2}$ it then follows that \cref{reseq} holds for general $P, Q$.
\end{proof}

\section{NilHecke Algebras}
\subsection{Towers of NilHecke Algebras}
Let $\As=\oplus_n \nh_n$, where $\nh_n$ is the graded $\kb$-algebra generated by $\kb[x_1, \ldots, x_n]$ and $\partial_1, \ldots, \partial_{n-1}$ where $|x_i|:=2$ and $|\partial_i|:=-2$, with the following relations:
\begin{align}
\nonumber
\partial_i^{\,2} = 0, 
\quad  
 \partial_i\,\partial_{i+1}\,\partial_i = \partial_{i+1}\,\partial_i\,\partial_{i+1},
 \quad
\partial_i\,\partial_j = \partial_j\,\partial_i,
\quad
 x_i\,\partial_j = \partial_j\,x_i,
 \quad
(\textup{for } |i-j|>1)
\\  
\label{slideq}
x_i\,\partial_i \;-\; \partial_i\,x_{i+1} = 1, 
\qquad \partial_i\,x_i \;-\; x_{i+1}\,\partial_i = 1.
\end{align}
Let us recall some facts about $\nh_n$ from \cite{KL09}.
Using the diagrammatic interpretation of $\nh_n$ therein, 
the map $\nh_n\otimes\nh_m\to \nh_{n+m}$ is given by stacking diagrams next to each other. 
\cite{KL09} shows that $\nh_{n+m}$ is a two sided projective $\nh_n\otimes \nh_m$ module. 
Moreover, there is an anti-automorphism $\psi$ of $\nh_n$ which flips diagrams upside down. The Grothendieck groups can be described as follows. 
\begin{prop}
\begin{enua}
    \item $K_\oplus(\nh_n)\cong\Zb[q^{\pm 1}] [P_n]$,
    where $P_n=\nh_n\psi(e_n)$.
    Here, $e_n$ is the idempotent $e_n := x_1^{n-1}x_2^{n-2}\ldots x_{n-1}\partial_{w_0}$ and $w_0 \in S_n$ is the longest element.
    \item $G_0^f(\nh_n) \cong \Zb[q^{\pm 1}] [L_n]$, 
    where $L_n=\kb[x_1, \ldots, x_n]/(\sym_n^+)$. Moreover, $\dim_\kb L_n=n!$. 
    The action of $x_i\in \nh_n$ on $L_n$ is given by multiplication by $x_i$, while the action of $\partial_i$ is given by the divided difference operator $\partial_i(f)=\frac{f-s_i(f)}{x_i-x_{i+1}}$.
    \end{enua}
\end{prop}
Thus, $\nh_n$ is a nice graded algebra with $I=\sym_m^+$. It follows that the nilHecke tower $\oplus_n \nh_n$ is a nice tower of algebras. Denote the symmetrized quantum binomial coefficients by and $\tbinomb{a}{b}_q := \prod_{i=0}^{b-1}\frac{[a-i]_q!}{[b-i]_q}$, where $[a]_q :=\frac{q^{a}-q^{-a}}{q-q^{-1}}$.
\begin{prop}[\cite{KL09}]\label{KLIND}
Let $\widetilde{P}_n=q^{-\frac{n(n-1)}{2}}P_n$. Then,
     $\ind(\widetilde{P}_n\boxtimes \widetilde{P}_m)
    %=\Ind_{\nh_n\boxtimes \nh_m}^{\nh_{n+m}}(\widetilde{P}_n\boxtimes \widetilde{P}_m )
    \cong \tbinomb{n+m}{n}_q \widetilde{P}_{n+m}$.
\end{prop}

\begin{prop}
\begin{enua}
\item For all $n,m$, there is an isomorphism $\ind(P_n\boxtimes P_m)=\tbinom{n+m}{n}_{q^{-2}} P_{n+m}$.
Consequently, $\sbrac{P_n}\mapsto t^{(n)}:=\frac{t^n}{(n)_{q^{-2}}!}$ defines an isomorphism 
\[
\cK(\As)\cong \Zb[q, q^{-1}][t^{(n)}| \, n\ge 1].
\]
%, where $v=q^{-2}$.
\item For all $n,m$, there is an isomorphism $\ind(L_n\boxtimes L_m)
%=\Ind_{\nh_n\boxtimes \nh_m}^{\nh_{n+m}}( L_n\boxtimes L_m )
\cong L_{n+m}$.
Consequently, $[L_n]\mapsto t^n$ defines an isomorphism    
\[
\cG^f(\As)\cong \Zb[q, q^{-1}][t].
\]
\end{enua}
\end{prop}
\begin{proof}
    Part (a) follows from combining \cref{KLIND} and the equality $\tbinomb{m+n}{n}_q= q^{mn}\tbinom{m+n}{n}_{q^{-2}}$.

    For (b), one checks that the map 
    \begin{align*}
    \nh_{n+m} \otimes_{\nh_n\boxtimes \nh_m} \scalemath{0.9}{\frac{\kb[x_1, \ldots, x_n]}{(\sym_n^+)}\boxtimes \frac{\kb[x_1, \ldots, x_m]}{(\sym_m^+)}}&\longrightarrow \scalemath{0.9}{\frac{\kb[x_1, \ldots, x_{n+m}]}{(\sym_{n+m}^+)}}  \\
    T\otimes (a\boxtimes b) &\longrightarrow T(ab)
    \end{align*} 
    is well-defined, surjective, and so by dimension reasons must be an isomorphism. 
\end{proof}

\subsection{Commutation Relations}
Let $P=P_1:=\kb[x_1]$,
and let $M=L_1:=\kb[x_1]/(x_1)=\kb$, the unique graded simple of $ \nh_1=\kb[x_1]$. Define functors
\[
\paren{\I_\kb}_n:=\I_\kb|_{\nh_n\-\fmod},
\quad\textup{and}\quad
\paren{\TR_{\kb[x_1]}}_n:=\TR_{\kb[x_1]}|_{\nh_n\-\fmod}.
\]

By definition, $\paren{\I_\kb}_n=\Ind_{ \nh_n \otimes \kb[x_{n+1}] }^{\nh_{n+1}}(-\boxtimes \kb)$ and $\paren{\TR_{\kb[x_1]}}_n=\Res_{\nh_{n-1}\otimes \kb[x_n]}^{\nh_{n}}(-)$. Thus, we see that:
\begin{itemize}
\item
The functor $\paren{\I_\kb}_n$ is represented by the $(\nh_{n+1}, \nh_n)$-bimodule $\nh_{n+1}\mathop\otimes\limits_{\nh_n  \otimes \kb[x_{n+1}] } (\nh_n\boxtimes \kb)$.
\item The functor $\paren{\TR_{\kb[x_1]}}_n$ \vspace{0.5ex}
     is represented by the 
     $(\nh_{n-1}, \nh_n)$-bimodule 
     $({}_{1\to n-1} \nh_n)$, which is $\nh_n$ regarded as a left $\nh_{n-1}$-module acting on strands $1$ through $n-1$.
\end{itemize}

\begin{example}
For $n=1$, $\nh_1=\kb[x_1]$. As $(\kb[x_1], \kb[x_1])$-bimodules,
\begin{equation}\label{resindeq}
\paren{\TR_{\kb[x_1]}\circ \I_\kb}_{1}
={}_{1\to 1}\nh_2\otimes_{\nh_1\boxtimes \kb[x_2]} (\nh_1 \otimes \kb)
= {}_{1\to 1}\nh_2\otimes_{\kb[x_1]\otimes  \kb[x_2]}  (\kb[x_1]  \boxtimes \kb).   
\end{equation}
Similarily, as $(\kb[x_1], \kb[x_1])$-bimodules, 
\[
\paren{\I_\kb\circ \TR_{\kb[x_1]}}_{1}
=\nh_1\otimes_{  \nh_{0}\otimes \kb[x_1] } ({}_{1\to 0}\nh_1  \boxtimes \kb )
= \kb[x_1]\otimes_{ \kb \otimes   \kb[x_1]} ({}_{1\to 0}\kb[x_1] \boxtimes \kb)    
=  {}^0 \kb[x_1], 
\]
where ${}^0 \kb[x_1] $ is the $(\kb[x_1], \kb[x_1])$-bimodule where $x_1$ acts by 0 on the left. 
Since $\nh_2$ has a $\kb$-basis $\set{x_1^ax_2^b}_{a,b\ge 0}\sqcup \set{\partial_1 x_1^ax_2^b}_{a,b\ge 0} $, 
 the right-hand side of \eqref{resindeq} has a $\kb$-basis 
 $S \sqcup Q$ given by
\[
S:=\set{1\otimes_{\kb[x_1]\boxtimes \kb[x_2]} x_1^b}_{b\ge 0},
\qquad 
Q:=\set{\partial_1 \otimes_{\kb[x_1]\boxtimes \kb[x_2]} x_1^b}_{b\ge 0}.
\]
Note that $S$ is a sub $(\kb[x_1], \kb[x_1])$-module isomorphic to $\kb[x_1]$ and we thus have a short exact sequence of $(\kb[x_1], \kb[x_1])$-bimodules: 
\[ 
0\to  \kb[x_1] 
\to \TR_{\kb[x_1]}\circ \I_\kb|_{\nh_1\-\fmod} \to Q\to 0.
\]
%\[
%\set{1\otimes_{\kb[x_1]\otimes \kb[x_2]} x_1^b}_{b\ge 0}
%\sqcup
%\set{\partial_1 \otimes_{\kb[x_1]\otimes \kb[x_2]} x_1^b}_{b\ge 0}.
%\]
%$S:=\textup{Span}_\kb\set{1\otimes_{\kb[x_1]\otimes \kb[x_2]} x_1^b|{b\ge 0}}$ 
% $Q:= \textup{Span}_\kb\set{\partial_1 \otimes_{\kb[x_1]\otimes \kb[x_2]} x_1^b|b\ge 0}$.
Using the relation $x_1\partial_1=\partial_1 x_2+1$, we see that $Q$ is isomorphic to ${}^0 \kb[x_1] $, without grading considerations.
Moreover, after taking grading into account,  $Q \cong q^{-2}    \paren{\I_\kb\circ \TR_{\kb[x_1]}}_{1}$  as graded $(\kb[x_1], \kb[x_1])$-bimodules. 
In general, we have that:
\end{example}

\begin{lemma}
\label{basecase}
As $(\nh_n, \nh_{n})$-bimodules,
    \begin{align*}
        \paren{\TR_{\kb[x_1]}\circ \I_\kb}_{n}
        &=
        {}_{1\to n}\nh_{n+1}
        \mathop\otimes\limits_{\nh_n\boxtimes \kb[x_{n+1}]}
        (\nh_n \boxtimes \kb), 
        \\
        \paren{\I_\kb\circ \TR_{\kb[x_1]}}_{n}
        &= \nh_{n}
        \mathop\otimes\limits_{  \nh_{ n-1} \boxtimes \kb[x_n] } 
        ({}_{1\to n-1}\nh_{ n}\boxtimes \kb). 
    \end{align*} 
Moreover, there is a nonsplit short exact sequence of graded $(\nh_n, \nh_{n})$-bimodules  
\[   
0\to  \nh_n \to \paren{\TR_{\kb[x_1]}\circ \I_\kb}_{n} \to q^{-2}  \paren{\I_\kb\circ \TR_{\kb[x_1]}}_{n}\to 0. 
\]
\end{lemma}
\begin{proof}
Left coset representatives of $S_n$ in $S_{n+1}$ are of the form $\set{e, s_n, s_{n-1}s_n, \ldots, s_1\ldots s_n}$. Thus, we obtain the following vector space decompositions:
\[ 
\paren{\TR_{\kb[x_1]}\circ \I_\kb}_{n}
=
\set{1\otimes_{\nh_n \boxtimes \kb[x_{n+1}]} \nh_n} 
\sqcup \ldots \sqcup  
\set{ \partial_1\ldots \partial_{n-1} \partial_n\otimes_{\nh_n \boxtimes \kb[x_{n+1}]} \nh_n },  \]
\[ 
\paren{\I_\kb\circ \TR_{\kb[x_1]}}_{n}
=\set{1\otimes_{\nh_{n-1} 
\boxtimes \kb[x_{n}]} \nh_n} 
\sqcup \ldots \sqcup  
\set{ \partial_1\ldots  \partial_{n-1}\otimes_{\nh_{n-1} \boxtimes \kb[x_{n}]} \nh_n }.  
\]
Note that $S:=\set{1\otimes_{\nh_n\boxtimes \kb[x_{n+1}]} \nh_n \boxtimes \kb}$ is a $(\nh_n, \nh_{n})$-submodule of $\paren{\TR_{\kb[x_1]}\circ \I_\kb}_{n}$ that is isomorphic to $\nh_n$. 
Therefore, there is an exact sequence:
\[ 
0\to \nh_n\to \paren{\TR_{\kb[x_1]}\circ \I_\kb}_{n} \to Q\to 0. 
\]
Define a map $ \Phi: \paren{\I_\kb\circ \TR_{\kb[x_1]}}_{n}\to Q$ by $v\otimes w\mapsto v\partial_n \otimes w$. 
We first check that this map is well defined. 
Let $v,w\in \nh_n$. 
Then, $vx_n\otimes  w=v\otimes x_n w=0$ in $\paren{\I_\kb\circ \TR_{\kb[x_1]}}_{n}$. 
Next, we compute that
\[
\Phi(vx_n\otimes w) 
= vx_n\partial_n \otimes w
= v \partial_n x_{n+1}\otimes w+v\otimes w
=v\partial_n \otimes x_{n+1}w+1\otimes vw
\sim 0  \ \textup{in} \ Q.   
\]
Similarly, for $n\in \nh_{n-1}$, $vn\otimes w=v\otimes n w $. 
We further check that
\[ 
\Phi(vn\otimes w)
=vn\partial_n \otimes w
=v\partial_n \otimes nw
=\Phi(v\otimes nw). 
\]
By construction, $\Phi$ is a morphism of $(\nh_n, \nh_{n})$-bimodules.
By the vector space decomposition above, $\Phi$ is bijective, and thus we have an isomorphism $q^{-2}\paren{\I_\kb\circ \TR_{\kb[x_1]}}_{n}\cong Q$. 
\end{proof}

\subsection{Higher Commutation Relations}

We will make use of the diagrammatics for $\nh_n$ developed in \cite{KLMS12}. Elements of the form $xy \in \nh_n$ are read from right to left, and drawn from bottom to top. For example, the element $x_1\partial_1$ is drawn as 
\[\begin{tikzpicture}[scale=0.5]
\draw (0,0) -- (2,2);
\draw (0,2) -- (2,0);
\filldraw (0.5,1.5) circle (3pt);
\end{tikzpicture}
 \]

\begin{lemma}
\label{idemlem}
For any $k\in \mathbb{Z}^+$,
    $e_k=x_1\ldots x_{k-1}\partial_{k-1}\ldots \partial_1 e_{k-1}$.
\end{lemma}
%TODO, proof

%For any element $b \in \nh_k$, denote by $\ell\set{b}$ the element obtained by %placing the diagram corresponding to $b$ on the last $k$ strands.
\begin{theorem}
    Let $\ell\set{e_k}$ be the idempotent $e_k$ placed on the last $k$ strands. As $(\nh_{n-k+1}, \nh_{n})$-bimodules 
    \begin{align*}
        \paren{\TR_{P_k}\circ \I_\kb}_{n}
        &={}_{1\to n+1-k} \, \ell\set{e_k}\nh_{n+1}
        \mathop\otimes\limits_{\nh_n\boxtimes \kb[x_{n+1}]} 
        (\nh_n \boxtimes \kb), 
        \\
         \paren{\I_\kb\circ \TR_{P_k}}_{n}
         &= \nh_{n+1-k}
         \mathop\otimes\limits_{  \nh_{ n-k} \boxtimes \kb[x_{n-k+1}] } 
         ({}_{1\to n-k} \, \ell\set{e_k}\nh_{ n}\boxtimes \kb). 
    \end{align*} 
Moreover, there is a nonsplit short exact sequence of graded $(\nh_{n-k+1}, \nh_{n})$-bimodules 
\[   
0\to \paren{\TR_{P_{k-1}}}_{n}  
\to \paren{\TR_{P_k}\circ \I_\kb}_{n} 
\to q^{-2k}  \paren{\I_\kb\circ \TR_{P_k}}_{n}
\to 0. 
\]
\end{theorem}
\begin{proof}
    Left coset representatives of $S_n$ in $S_{n+1}$ are of the form $\set{e, s_n, s_{n-1}s_n, \ldots, s_1\ldots s_n}$. However, note that 
    \begin{equation}\label{lzeroeq}
        \ell\set{\partial_{w_0}(k)} \partial_i=0 
        \quad \textup{for all}\quad   n+2-k \le i\le n, 
    \end{equation}
    as the corresponding expression in $S_{n+1}$ is not reduced. It follows that
    \begin{equation}\label{eq:lek1neq}
        \ell\set{e_k} \partial_i \partial_{i+1}\ldots \partial_n=0 
    \quad \textup{for all}\quad   n+2-k \le i\le n. 
    \end{equation}  
    On the other hand, 
    $\ell\set{\partial_{w_0}(k)} \partial_r \partial_{r+1}\ldots \partial_n\neq 0 $ if $ 1\le r\le n+1-k$, as the corresponding expression in $S_{n+1}$ is reduced. 
    Thus,
    \begin{equation}\label{eq:lekneq}
    \ell\set{e_k} \partial_r \partial_{r+1}\ldots \partial_n\neq 0
    \quad \textup{for all}\quad  1\le r\le n+1-k.    
    \end{equation}
    For brevity, we denote
    $\otimes^\prime :=\otimes_{\nh_n \boxtimes \kb[x_{n+1}]}$. 
    From \eqref{eq:lek1neq} and \eqref{eq:lekneq} we obtain the following vector space decompositions:
    \begin{multline*}       
    \paren{\TR_{P_k}\circ \I_\kb}_{n}
    =
    %\scalemath{0.9}{
    \set{\ell\set{e_k}\otimes^\prime \nh_n\boxtimes \kb } \sqcup  
    \\
    \set{ \ell\set{e_k}\partial_{n+1-k}\ldots \partial_n\otimes^\prime\nh_n\boxtimes \kb  }
    \sqcup  \ldots \sqcup  
    \set{ \ell\set{e_k}\partial_1\ldots \partial_n\otimes^\prime\nh_n \boxtimes \kb  }.
    %}  
    \end{multline*}
    \begin{multline*}    
    \paren{\I_\kb\circ \TR_{P_k}}_{n} 
    = 
    %\scalemath{0.9}{ 
    \set{1\otimes_{  \nh_{ n-k} \boxtimes \kb[x_{n-k+1}] }\ell\set{e_k}\nh_n\boxtimes \kb } 
    \\
    \sqcup \ldots \sqcup  
    \set{ \partial_1\ldots  \partial_{n-k}\otimes_{  \nh_{ n-k} \boxtimes \kb[x_{n-k+1}] }\ell\set{e_k} \nh_n\boxtimes \kb  }.
    %} \vspace{1ex}
    \end{multline*}
    Note that $S:=\set{ \ell\{e_{k}\}\otimes^\prime \nh_n\boxtimes \kb }$ is a sub $(\nh_{n-k+1}, \nh_{n})$-bimodule of $\paren{\TR_{P_k}\circ \I_\kb}_{n}$. \vspace{1ex} \\
    \textbf{Claim: } $S\cong \paren{\TR_{P_{k-1}}}_{n}= {}_{1\to n-(k-1)}\ell\set{e_{k-1}}\nh_n$. Diagrammatically, this is represented by
    \begin{equation}
    \label{ekeq}
\vcenter{\hbox{
        \begin{tikzpicture}[scale=0.35,line width=0.8pt,>=stealth]
  % three vertical wires with dots (fixed positions)
  \draw (0,4) -- (0,2);
  \draw ( -3.5,4) -- ( -3.5,2);
  \draw (0,0.7) -- (0,-1);
  \draw ( -3.5,0.7) -- ( -3.5,-1);

  % dots under ek box
  \node at (-2,-0.3) {$\cdots$};
  \node at (-2,3) {$\cdots$};
  % left output box A_n and connections
  \draw (-6,-3) rectangle (0.7,-1);
  \node at (-1.8,-2) {$\scalemath{1.1}{\nh_n}$};
  \draw (-2,-1) -- (-2,-1);

  % right output label k and connection
  \node at (2,-2.2) {$\scalemath{1.2}{\boxtimes \,  \kb}$};
  \draw (2.4,0.7) -- (2.4,-1.5);
  \draw (2.4,4) -- (2.4,2);

  % ek box
  \draw (-4,2) rectangle (2.8,0.7);
  \node at (-0.5,1.3) {$e_{k}$};
\end{tikzpicture} \raisebox{6ex}{ $\ \cong $}\begin{tikzpicture}[scale=0.35,line width=0.8pt,>=stealth]

  % three vertical wires with dots (fixed positions)
  \draw (0,4) -- (0,2);
  \draw ( -3.5,4) -- ( -3.5,2);
  \draw (0,0.7) -- (0,-1);
  \draw ( -3.5,0.7) -- ( -3.5,-1);

  % dots under ek box
  \node at (-2,-0.3) {$\cdots$};
  \node at (-2,3) {$\cdots$};
  % left output box A_n and connections
  \draw (-6,-3) rectangle (0.7,-1);
  \node at (-1.8,-2) {$\scalemath{1.1}{\nh_n}$};
  \draw (-2,-1) -- (-2,-1);

  % right output label k and connection
  \node at (2,-2.2) {$\scalemath{1.2}{\boxtimes \,  \kb}$};
  \draw (2.4,4) -- (2.4,-1.5);

  % ek box
  \draw (-4,2) rectangle (1.5,0.7);
  \node at (-0.5,1.3) {$e_{k-1}$};
\end{tikzpicture}}}
    \end{equation}
We proceed by induction on $k$, where the base case is handled by \cref{basecase}. 
By \cref{idemlem}, the left-hand side of \eqref{ekeq} is the leftmost diagram below:  
\[
\begin{tikzpicture}[scale=0.35,line width=0.8pt,>=stealth]
  % three vertical wires with dots (fixed positions)
  \draw (-2,4) -- (-2,2);
  \draw ( 0,4) -- ( 0,2);
  \draw ( 2.4,4) -- ( 2.4,2);
  \draw (-2,0.7) -- (-2,-1);
  \draw ( 0,0.7) -- ( 0,-1);
  \draw ( 2.4,0.7) -- ( 2.4,-1);

  \fill (-2,3.5) circle (5pt);
  \fill ( 0,3.5) circle (5pt);
  \fill ( 2.4,3.5) circle (5pt);

  % dots under ek box
  \node at (-1,0) {$\cdots$};
  \node at (-1,2.4) {$\cdots$};
  % left output box A_n and connections
  \draw (-6,-3) rectangle (0.7,-1);
  \node at (-1.8,-2) {$\scalemath{1.1}{\nh_n}$};
  \draw (-2,-1) -- (-2,-1);

  % right output label k and connection
  \node at (2,-2.2) {$\scalemath{1.2}{\boxtimes \,  \kb}$};
  \draw (2.4,-1) -- (2.4,-1.5);

  % curved wire: starts at left top dot, goes up & left of e_k,
  % then runs right below the three dots, and turns back up.
  % (It does NOT touch the third dot.)
  \draw (-4,-1)
        .. controls (-3.6,2.6) .. (-3.5,2.7) 
        .. controls (-1.0,3.0) and ( 1.0,3.0) .. ( 3.3,3.0)
        -- (3.3,4);
  % ek box
  \draw (-3,2) rectangle (3,0.7);
  \node at (0,1.3) {$e_{k-1}$};
\end{tikzpicture}
\raisebox{6ex}{ \, $\xequal{ind. hyp.}$}\begin{tikzpicture}[scale=0.35,line width=0.8pt,>=stealth]

  % three vertical wires with dots (fixed positions)
  \draw (-2,4) -- (-2,2);
  \draw ( 0,4) -- ( 0,2);
  \draw ( 2.4,4) -- ( 2.4,2);
  \draw (-2,0.7) -- (-2,-1);
  \draw ( 0,0.7) -- ( 0,-1);
  \draw ( 2.4,0.7) -- ( 2.4,-1);

  \fill (-2,3.5) circle (5pt);
  \fill ( 0,3.5) circle (5pt);
  \fill ( 2.4,3.5) circle (5pt);

  % dots under ek box
  \node at (-1,0) {$\cdots$};
  \node at (-1,2.4) {$\cdots$};
  % left output box A_n and connections
  \draw (-6,-3) rectangle (0.7,-1);
  \node at (-1.8,-2) {$\scalemath{1.1}{\nh_n}$};
  \draw (-2,-1) -- (-2,-1);

  % right output label k and connection
  \node at (2,-2.2) {$\scalemath{1.2}{\boxtimes \,  \kb}$};
  \draw (2.4,3) -- (2.4,-1.5);

  % curved wire: starts at left top dot, goes up & left of e_k,
  % then runs right below the three dots, and turns back up.
  % (It does NOT touch the third dot.)
  \draw (-4,-1)
        .. controls (-3.6,2.6) .. (-3.5,2.7) 
        .. controls (-1.0,3.0) and ( 1.0,3.0) .. ( 3.3,3.0)
        -- (3.3,4);
  % ek box
  \draw (-3,2) rectangle (1.8,0.7);
  \node at (-0.5,1.3) {$e_{k-2}$};
\end{tikzpicture}
\raisebox{6ex}{ $\! \xequal{\eqref{slideq}}$}\begin{tikzpicture}[scale=0.35,line width=0.8pt,>=stealth]

  % three vertical wires with dots (fixed positions)
  \draw (-2,4) -- (-2,2);
  \draw ( 0,4) -- ( 0,2);
  \draw (-2,0.7) -- (-2,-1);
  \draw ( 0,0.7) -- ( 0,-1);

  \fill (-2,3.5) circle (5pt);
  \fill ( 0,3.5) circle (5pt);

  % dots under ek box
  \node at (-1,0) {$\cdots$};
  \node at (-1,2.4) {$\cdots$};
  % left output box A_n and connections
  \draw (-6,-3) rectangle (0.7,-1);
  \node at (-1.8,-2) {$\scalemath{1.1}{\nh_n}$};
  \draw (-2,-1) -- (-2,-1);

  % right output label k and connection
  \node at (2,-2.2) {$\scalemath{1.2}{\boxtimes \,  \kb}$};
  \draw (3,4) -- (2.4,-1.5);

  % curved wire: starts at left top dot, goes up & left of e_k,
  % then runs right below the three dots, and turns back up.
  % (It does NOT touch the third dot.)
  \draw (-4,-1)
        .. controls (-3.6,2.6) .. (-3.5,2.7) 
        .. controls (-1.0,3.0) and ( 1.0,3.0) .. ( 1.8,3.0)
        -- (1.8,4);
  % ek box
  \draw (-3,2) rectangle (1.8,0.7);
  \node at (-0.5,1.3) {$e_{k-2}$};
\end{tikzpicture}\raisebox{6ex}{ $\, =$}\begin{tikzpicture}[scale=0.35,line width=0.8pt,>=stealth]

  % three vertical wires with dots (fixed positions)
  \draw (0,4) -- (0,2);
  \draw ( -3.5,4) -- ( -3.5,2);
  \draw (0,0.7) -- (0,-1);
  \draw ( -3.5,0.7) -- ( -3.5,-1);

  % dots under ek box
  \node at (-2,-0.3) {$\cdots$};
  \node at (-2,3) {$\cdots$};
  % left output box A_n and connections
  \draw (-6,-3) rectangle (0.7,-1);
  \node at (-1.8,-2) {$\scalemath{1.1}{\nh_n}$};
  \draw (-2,-1) -- (-2,-1);

  % right output label k and connection
  \node at (2,-2.2) {$\scalemath{1.2}{\boxtimes \,  \kb}$};
  \draw (2.4,4) -- (2.4,-1.5);

  % ek box
  \draw (-4,2) rectangle (1.5,0.7);
  \node at (-0.5,1.3) {$e_{k-1}$};
\end{tikzpicture}
\]
Here, the last equality is a result of \cref{idemlem}. 
It then follows that there is an exact sequence
\[   
0\to \paren{\TR_{P_{k-1}}}_{n}  \to \paren{\TR_{P_k}\circ \I_\kb}_{n} \to Q\to 0 .
\]
Define a map $ \Phi: \paren{\I_\kb\circ \TR_{P_k}}_{n}\to Q$ as follows, for $v\in \nh_{n+1-k}, w\in \nh_k$:
\[
\Phi(v\otimes \ell\set{e_k} w) := \ell\set{e_k}  v \partial_{n+1-k}\ldots \partial_n \otimes^\prime w.    
\]
We first show that this map is well defined. 
Note $vx_{n-k+1}\otimes  \ell\set{e_k} w=v\otimes x_{n-k+1}  \ell\set{e_k} w=0$ in $\paren{\I_\kb\circ \TR_{P_k}}_{n}$. We compute that
\begin{equation}   
\begin{aligned}
    \Phi(vx_{n-k+1}\otimes  \ell\set{e_k} w)
    &= \ell\set{e_k}  v  x_{n-k+1}\partial_{n+1-k}\ldots \partial_n \otimes^\prime w 
    \\
    &\xequal{\eqref{slideq}} \ell\set{e_k}  v  \partial_{n+1-k} x_{n-k+2}\ldots \partial_n \otimes^\prime w 
    + \ell\set{e_k}  v  \partial_{n+2-k} \ldots \partial_n \otimes^\prime w
    \\
    &\xequal{\eqref{lzeroeq}} \ell\set{e_k}  v  \partial_{n+1-k} x_{n-k+2}\ldots \partial_n \otimes^\prime w.
\end{aligned}
\end{equation}
Here, we first commute the $v\in \nh_{n+1-k}$ past $\partial_{n+2-k}$ before applying \eqref{lzeroeq}. 
We continue moving $x_\bullet$ all the way to the right, and we see that 
\begin{multline*}   
    \Phi(v x_{n-k+1}\otimes  \ell\set{e_k} w)=  \ell\set{e_k}  v  \partial_{n+1-k}\ldots \partial_{n-1} \otimes^\prime w
    \\
    +\ell\set{e_k}  v  \partial_{n+1-k}\ldots \partial_n   x_{n+1}\otimes^\prime w \sim 0 
    \quad \textup{in} \quad Q.
\end{multline*} 
Similarly, for $n\in \nh_{n-k}$, $vn\otimes  \ell\set{e_k} w= v\otimes n \ell\set{e_k} w= v\otimes  \ell\set{e_k} nw$, and we check that
\begin{multline*}  
 \Phi(vn\otimes  \ell\set{e_k} w) 
 = \ell\set{e_k}  v n \partial_{n+1-k}\ldots \partial_n   \otimes^\prime w
 \\
 =   \ell\set{e_k}  v  \partial_{n+1-k}\ldots \partial_n   \otimes^\prime nw
 = \Phi(v\otimes  \ell\set{e_k} nw).
\end{multline*}
Because $\nh_{n-k+1}$ commutes with $\ell\set{e_k}\nh_{n+1}$, 
$\Phi$ is a morphism of $(\nh_{n-k+1}, \nh_{n})$-bimodules. 
It remains to show that $\Phi$ is a bijection. 
Using \cite[Lemma 2.2]{KLMS12} with $a=k$ and $b=1$, 
we see that
\[ 
\set{ \ell\set{e_k}\partial_{r}\ldots \partial_n\otimes^\prime\nh_n\boxtimes \kb  }= \set{ \ell\set{e_k}\partial_{r}\ldots \partial_n\otimes^\prime \ell\set{e_k}\nh_n\boxtimes \kb  } 
\quad \textup{for all}\quad 1\le r\le n+1-k. 
\]
Now, the bijectiveness follows from the vector space decompositions above. 
As $|\partial_i| = -2$, % just to eliminate an overfull line
the degree of $\partial_{n+1-k}\ldots \partial_n $ is $-2k$, 
and it follows that $Q\cong q^{-2k} \paren{\I_\kb\circ \TR_{P_k}}_{n}$. This completes the proof.
\end{proof}

Our results can be summarized as follows:

\begin{theorem}
The functors $\I_\kb$ and $\TR_{P_k}, k\ge 1$ acting on finite-dimensional modules over the nilhecke algebra, give a categorical realization of quantum differential operators on $\bbA^1_{\Zb[v,v^{-1}]}$. Specifically,
\begin{enua}
    \item Let $v=q^{-2}$. The functors $\I_\kb: \As\-\fmod\to \As\-\fmod$ and $\TR_{P_k}: \As\-\fmod \to \As\-\fmod$ satisfy the following categorified $v-$divided power Weyl algebra relations: 
    \[ 
    \I_\kb\circ \I_\kb\cong \I_{\Ind(\kb\boxtimes \kb)},
    \quad 
    \TR_{P_\ell}\circ \TR_{P_k}\cong \TR_{P_{\ell+k} }^{\oplus \binom{\ell+k}{k}_v}.
    \]
     \[   0\to \paren{\TR_{P_{k-1}}}  \to \paren{\TR_{P_k}\circ \I_\kb} \to v^k  \paren{\I_\kb\circ \TR_{P_k}}\to 0 \qquad \forall k\ge 1 \]
    \item The isomorphism $\cG^f(\As)\cong \Zb[v,v^{-1}][t]$ intertwines the action of $[\I_\kb]$, $[\TR_{P_k}], k\ge 1$ with right multiplication by $t$ and $\partial^{(k)}$, respectively. 
\end{enua}
\end{theorem}

%%%%%%%%%%%%%%%%%%%%%%%%%%%%%%%%%%%%%%%%%%%%%%%%%

%\newpage 

\printbibliography
\end{document}